\documentclass[11pt]{article}
       \usepackage{amsfonts}
       \usepackage{stmaryrd}
       \usepackage{latexsym,amssymb,mathrsfs,fancyhdr}
       \font\tenmsb=msbm10
       \font\sevenmsb=msbm7
       \font\fivemsb=msbm5
       \catcode`\@=11
       \ifx\amstexloaded@\relax\catcode`\@=\active
       \endinput\else\let\amstexloaded@\relax\fi
       \def\spaces@{\space\space\space\space\space}
       \def\spaces@@{\spaces@\spaces@\spaces@\spaces@\spaces@}
       \def\space@.  {\futurelet\space@\relax}
       \space@.
       \def\Err@#1{\errhelp\defaulthelp@\errmessage{AmS-TeX error: #1}}
       \def\relaxnext@{\let\next\relax}
       \def\accentfam@{7}
       \def\noaccents@{\def\accentfam@{0}}
       \def\Cal{\relaxnext@\ifmmode\let\next\Cal@\else
       \def\next{\Err@{Use \string\Cal\space only in math mode}}\fi\next}
       \def\Cal@#1{{\Cal@@{#1}}}
       \def\Cal@@#1{\noaccents@\fam\tw@#1}
       \def\Bbb{\relaxnext@\ifmmode\let\next\Bbb@\else
       \def\next{\Err@{Use \string\Bbb\space only in math mode}}\fi\next}
       \def\Bbb@#1{{\Bbb@@{#1}}}
       \def\Bbb@@#1{\noaccents@\fam\msbfam#1}
       \newfam\msbfam
       \textfont\msbfam=\tenmsb
       \scriptfont\msbfam=\sevenmsb
       \scriptscriptfont\msbfam=\fivemsb
\usepackage{dsfont}
\usepackage[german,english]{babel}
\usepackage{amsmath,amssymb}
\usepackage[square, comma, sort&compress, numbers]{natbib}
\usepackage{cases}
\usepackage{mathrsfs}
\usepackage{amsmath}
\usepackage{amsthm}
\usepackage{amsfonts}
\usepackage{amssymb}
\usepackage{latexsym}
\usepackage{fancyhdr}
\usepackage{extarrows}
\usepackage{geometry}
\usepackage{color}
\usepackage[comma,numbers,square,sort&compress]{natbib}
\usepackage{epstopdf}
\usepackage{graphicx,amsmath}
\usepackage{subfigure}
\usepackage{amssymb}
\theoremstyle{plain}
\newtheorem{theorem}{Theorem}[section]

\newtheorem{lemma}[theorem]{Lemma}
\newtheorem{example}[theorem]{Example}

\numberwithin{equation}{section}
\newtheorem{remark}[theorem]{Remark}
\newtheorem{definition}[theorem]{Definition}

\usepackage{enumerate}

\begin{document}

\numberwithin{equation}{section}
\title{Further   properties and representations of \\ the $W$-weighted  $m$-weak group inverse}

\author{{\ Jiale Gao$^{a}$,    Qing-Wen Wang$^{a,b}$\footnote{E-mail address: wqw@t.shu.edu.cn  (Q.W. Wang).},   Kezheng Zuo$^{c}$, }
\\ {\small$^{a}$Department of Mathematics and Newtouch Center for Mathematics,}\\{\small  Shanghai University, Shanghai 200444, China}
\\
{{\small $^b$Collaborative Innovation Center for the Marine Artificial Intelligence, Shanghai 200444, China}}
\\{{\small $^c$School of Mathematics and Statistics, Hubei Normal University, Huangshi 435002, China}}
}

\maketitle

{\noindent{\bf Abstract:}
The purpose of this paper is to explore   more  properties and representations of the $W$-weighted $m$-weak group (in short, $W$-$m$-WG) inverse.  We  first explore an interesting relation  between two projectors with respect to  the $W$-$m$-WG inverse.
 Then, the $W$-$m$-WG inverse is represented by various generalized inverses including $W$-weighted Drazin inverse,    $W$-weighted weak group inverse, $W$-weighted core inverse, etc. We also give three concise explicit expressions   for the $W$-$m$-WG inverse. Moreover, a canonical form of the $W$-$m$-WG inverse is presented in terms of the singular value decomposition. Finally, several numerical examples  are designed to illustrate some results given in the paper.
}

\begin{description}
\item[{\bf Key words}:] {$W$-weighted $m$-weak group inverse,  $m$-weak group inverse,  singular value decomposition}
\item[{\bf AMS subject classifications}:] 15A09, 15A24
\end{description}

{\section{Introduction}}\label{introductionsection}
Since Penrose  \cite{PenroseMpire} defined the Moore-Penrose   inverse using  four matrix equations in 1955, numerous papers on different topics of  generalized  inverses and their applications have been published \cite{Benbookre,Draganabookre,Wangbookre}.
In 1958, Drazin \cite{Drazininversere} introduced the Drazin inverse in associative rings and semigroups. Owing to  its significant spectral properties,  it is now utilized extensively in diverse areas of research.
A special case of the Drazin inverse is called the  group inverse studied by
Erdelyi \cite{Erdegroupinversere} in 1967.
Cline and Greville \cite{wdrazinclinere} established the  $W$-weighted Drazin  inverse in 1980 through a weight matrix, which generalizes the concept of the Drazin inverse.

The core inverse  was first proposed by Baksalary and Trenkler \cite{Bakscoreinversere} in 2010 as an alternative to the group inverse.
Manjunatha Prasad and Mohana \cite{PrasadacoreEPinversere} then extended the core inverse in 2014 by defining  the core-EP inverse.
 In 2018, Ferreyra et al. \cite{wcoreepiFerre}   proposed the weighted core-EP decomposition  to introduce  the weighted core-EP inverse.

Wang et al. \cite{Wangwaekgropinverre} introduced the weak group inverse in 2018 because the generalizations of the group inverse don't receive as much attention as that of the core inverse.
Ferreyra et al. \cite{wweakgroupFerre}  defined the $W$-weighted weak group inverse in 2019 to extend  the weak group inverse.

In 2021,  Zhou et al. \cite{wmweakgroupringZhoure} introduced the $m$-weak group inverse in an   involutive ring, which coincides with the  weak group inverse as $m=1$.
Soon afterward,  Jiang and Zuo \cite{wmweakgroupJiare} studied  further characterizations for the $m$-weak group inverse of a complex matrix.
 In 2023, Ferreyra and  Malik \cite{GGinverseFerryre} introduced the   generalized group  inverse, which is the  $m$-weak group inverse  for $m=2$ in essence.
 Mosi\'{c} et al. \cite{Mosinewmwgire,MosicmWGapp} and Li et al. \cite{LIringmwginversere} investigated  more properties  and applications for the $m$-weak group inverse in  complex matrix sets and     involutive rings, respectively.

 Recently,   Mosi\'{c} et al. \cite{MosicWmWG} researched   the $W$-weighted $m$-weak group (in short, $W$-$m$-WG) inverse, which  is a generalization of the  $W$-weighted weak group inverse,   $W$-weighted Drazin inverse and  $m$-weak group inverse. Furthermore,
Mosi\'{c} et al.  \cite{MosicWmWGapp}   considered   applications  of the $W$-$m$-WG inverse in solving a constrained minimization problem, which extend the main results in {\cite{MosicmWGapp,wangwgappapprore}.

Motivated by the work of Mosi\'{c} et al. for  the $W$-$m$-WG inverse,
 our main  objective in this paper   is  to show more interesting  properties and representations of the $W$-$m$-WG inverse. Our  contributions  are briefly listed  below.
  \begin{enumerate}[$(1)$]
   \item   We give two projectors related to the $W$-$m$-WG inverse (see Theorem \ref{wmwgproth}), by which an interesting  property   for the $W$-$m$-WG inverse is deduced
        (see Theorem \ref{WAWAWmWeqWaWmWAeq}).

   \item  Various representations of the $W$-$m$-WG inverse  are proposed by using different generalized inverses, such as  the $W$-weighted Drazin inverse,    $W$-weighted weak group inverse, $W$-weighted core inverse, etc. (see Theorems \ref{represwmwgi01th} and \ref{corechawmwgth}).  Moreover, we    express the $W$-$m$-WG inverse through  three formulas  with concise and neat forms (see Theorem \ref{conciseforth}).

   \item Finally,  a canonical form of the $W$-$m$-WG inverse   is given by the singular value decomposition (see Theorem \ref{WmWGinHSdecth}).
 \end{enumerate}

The section arrangement of this paper  is as follows.   We review some notations, definitions and lemmas in Section \ref{Prelsection}. We devote   Section \ref{resultsection} to providing some  properties and representations of the $W$-$m$-WG inverse. Several numerical examples in Section \ref{numbersection} are shown to test   most of  the results in Section \ref{resultsection}.
Conclusions  are  stated in Section \ref{Conclusionssection}.

\section{Preliminaries}\label{Prelsection}
We recall some necessary  notations, definitions and lemmas.
Let $\mathbb{C}^{q\times n}$  and $\mathbb{Z}^{+}$ be the sets of all $q\times n$ complex matrices   and  all positive integers, respectively.
 For $A\in\mathbb{C}^{q\times n}$, the symbols $A^*$, ${\rm rank}(A)$,  ${\rm Ind}(A)$ (for $q=n$), $\mathcal{R}(A)$, $\mathcal{N}(A)$ and $\Vert A \Vert_F$ denote the conjugate transpose,  rank, index, {range space}, null space  and Frobenius  norm of $A$, respectively.
 For $A\in\mathbb{C}^{n\times n}$, we  define   $A^{0}=I_n$, where  $I_n$ is the identity matrix  of order $n$. And, $0$ denotes  the    null matrix with  an appropriate size.
For two subspaces $\mathcal{T}$ and  $\mathcal{S}$   satisfying their direct sum is a $n$-dimensional complex vector space $\mathbb{C}^{n}$, the
projector onto  $\mathcal{T}$   along  $\mathcal{S}$ is denoted by  $P_{\mathcal{T},\mathcal{S}}$.
And,    $P_{\mathcal{T}}=P_{\mathcal{T},\mathcal{T}^{\perp}}$, where $\mathcal{T}^{\perp}$ denotes the orthogonal complement of $\mathcal{T}$.
For $A\in\mathbb{C}^{q\times n}$ and a subspace  $\mathcal{W}$ of $\mathbb{C}^{n }$,
we define   $A\mathcal{W}=\{Ax ~|~x\in\mathcal{W}\} \subseteq \mathbb{C}^{q}$.

Let $A\in\mathbb{C}^{q\times n}$, $B\in\mathbb{C}^{q\times n}$ and $W(\neq 0)\in\mathbb{C}^{n\times q}$. We define the $W$-product of $A$ and $B$   by $  A\star B= A W B$ (see \cite{wcoreepiFerre}). It is well known that if  $\Vert A\Vert_W:=\Vert A\Vert \Vert W\Vert $, where $\Vert \cdot  \Vert $ is  a  matrix norm on $\mathbb{C}^{q\times n}$, then
 $\left(  \mathbb{C}^{q\times n}, \star,  \Vert \cdot \Vert_W\right)$ is a Banach algebra.
Moreover,  we define the $W$-product of $A$ with itself $l\in\mathbb{Z}^{+}$ times by $A^{\star l}$, i.e.,
\begin{equation*}
A^{\star l}=  \underbrace {A \star A\star... \star A}_{l \text{ times}} .
\end{equation*}
 Specially, we define $A^{\star 0} W=I_{q}$ and $WA^{\star 0}  =I_{n}$.

After that, we introduce some generalized inverses.
Let $A\in\mathbb{C}^{q\times n}$. Then  the Moore-Penrose inverse of $A$ (see \cite{PenroseMpire}), denoted by $A^{\dag}$, is the unique matrix $X\in\mathbb{C}^{n\times q}$ satisfying
\begin{equation*}
 AXA=A,~ XAX=X,~(AX)^*=AX,~(XA)^*=XA.
\end{equation*}
Let $A\{2\}=\{ X\in\mathbb{C}^{n\times q} ~|~ XAX=X\}$.
If  $X\in A\{2\} $ satisfies $\mathcal{R}(X)=\mathcal{T}$ and $\mathcal{N}(X)=\mathcal{S}$ for two subspaces $\mathcal{T}\subseteq\mathbb{C}^{n}$ and $\mathcal{S}\subseteq\mathbb{C}^{q}$,  then $X$ is called the outer inverse with the prescribed {range space} $\mathcal{T}$ and   null space $\mathcal{S}$, denoted by $A^{(2)}_{\mathcal{T},\mathcal{S}}$ (see \cite{Wangbookre}).

\begin{definition}\label{vardenfth}
Let $A\in\mathbb{C}^{q\times n}$, $W(\neq 0)\in\mathbb{C}^{n\times q}$ and $k=\max\{{\rm Ind}(AW),{\rm Ind}(WA)\}$.
\begin{enumerate}[$(1)$]
  \item
The $W$-weighted Drazin inverse of $A$ (see \cite{wdrazinclinere}), denoted by $A^{D,W}$, is the unique matrix $X$ satisfying
\begin{equation*}
(AW)^k=(AW)^{k+1}XW,~X=XWAWX,~AWX=XWA.
\end{equation*}
If $W=I_n$, then $A^{D,W}=A^{D}$, which is the Drazin inverse of $A$ (see \cite{Drazininversere}).
If $k \leq 1$, then $A^{D,W}=A^{\#,W}$, which is the $W$-weighted group inverse.
If $W=I_n$ and $k \leq 1$, then  $A^{D,W}=A^{\#}$, which is the group inverse of $A$ (see \cite{Erdegroupinversere}).
\item The weighted core-EP inverse
of $A$ (see \cite{wcoreepiFerre}), denoted by $A^{\textcircled{\dag},W}$, is the unique matrix satisfying
\begin{equation*}
 WAWX=P_{\mathcal{R}((WA)^k)} ,~ \mathcal{R}(X) \subseteq  \mathcal{R}((AW)^k).
\end{equation*}
If $W=I_n$, then $A^{\textcircled{\dag},W}=A^{\textcircled{\dag}}$, which is the core-EP inverse of $A$ (see \cite{PrasadacoreEPinversere}). If $k \leq 1$, then $A^{\textcircled{\dag},W}=A^{\textcircled{\#},W}$, which is the weighted  core inverse of $A$ (see \cite{wweakgroupFerre}). If $W=I_n$ and $k \leq 1$, then $A^{\textcircled{\dag},W}=A^{\textcircled{\#}}$, which is the core inverse of $A$ (see \cite{Bakscoreinversere}).
\item The  $W$-weighted weak group   inverse
of $A$ (see \cite{wweakgroupFerre}), denoted by $A^{\textcircled{w},W}$, is the unique matrix satisfying
\begin{equation*}
 AWX^{\star 2}=X,~AWX=A^{\textcircled{\dag},W}WA.
\end{equation*}
If $W=I_n$, then $A^{\textcircled{w},W}=A^{\textcircled{w}}$, which is the weak group inverse of $A$ (see \cite{Wangwaekgropinverre}).

\item\label{vardenfitem04} Let $m\in\mathbb{Z}^{+}$. The $W$-weighted $m$-weak group (in short, $W$-$m$-WG ) inverse (see \cite{MosicWmWG}), denoted by $A^{\textcircled{w}_m,W}$, is defined as
\begin{equation}\label{Wmwgmpmaineq}
  A^{\textcircled{w}_m,W}=(A^{\textcircled{\dag},W})^{\star (m+1)}WA^{\star m}.
\end{equation}
If $W=I_n$, then $A^{\textcircled{w}_m,W}=A^{\textcircled{w}_m}$, which is $m$-weak group inverse of $A$ (see \cite{wmweakgroupJiare}). If  $W=I_n$ and  $m=2$, then $A^{\textcircled{w}_m,W}=A^{\textcircled{w}_2}$, which is the generalized group inverse of $A$ (see \cite{GGinverseFerryre}). Particularly, if $m=1$, then $  A^{\textcircled{w}_m,W}=  A^{\textcircled{w},W}$, and if $k\leq m$, then $  A^{\textcircled{w}_m,W}=  A^{D,W}$ (see \cite{MosicWmWG}).
\end{enumerate}
\end{definition}

\begin{lemma}\label{wcoreepprole}
Let $A\in\mathbb{C}^{q\times n}$, $W(\neq 0)\in\mathbb{C}^{n\times q}$ and $k=\max\{{\rm Ind}(AW),{\rm Ind}(WA)\}$.  Then:
\begin{enumerate}[$(1)$]
  \item\label{wcoreepimaitem1}\cite[Theorems 2.2 and 2.3]{wcoreepigaore}
  $A^{\textcircled{\dag},W}=A((WA)^{\textcircled{\dag}})^2$;
    \item\label{wcoreepimaitem3} $A^{\textcircled{\dag},W}WAW=P_{\mathcal{R}((AW)^k),
        \mathcal{N}\left(((WA)^k)^{*}WAW\right)}$;
    \item\label{wcoreepimaitem7}
    $(A^{\textcircled{\dag},W})^{\star s}=(A^{\star s})^{\textcircled{\dag},W}$, where $s\in\mathbb{Z}^{+}$.
\end{enumerate}
\end{lemma}

\begin{proof}
\eqref{wcoreepimaitem3}
Note that if $B\in\mathbb{C}^{q\times n}$, $C\in\mathbb{C}^{n\times g}$, $D\in\mathbb{C}^{p\times n}$ and $\mathcal{N}(B)=\mathcal{N}(D)$, then
\begin{equation*}
    \mathcal{N}(BC)= (C^*\mathcal{N}^{\perp}(B))^{\perp}=
    (C^*\mathcal{N}^{\perp}(D))^{\perp}  =
   \mathcal{N}(DC).
\end{equation*}
So, it follows from \cite[Proposition 3.1(2)]{wcoreepigaore}, i.e., $\mathcal{N}(A^{\textcircled{\dag},W})=\mathcal{N}\left(((WA)^k)^*\right)$, that
 \begin{equation*}
   \mathcal{N}(A^{\textcircled{\dag},W}WAW)=\mathcal{N}(((WA)^k)^*WAW).
 \end{equation*}
Thus, using \cite[Proposition 3.1(1)]{wcoreepigaore}, i.e., $\mathcal{R}(A^{\textcircled{\dag},W})=\mathcal{R}((AW)^k)$, and  \cite[Theorem 2.6]{wcoreepigaore}, i.e.,
$A^{\textcircled{\dag},W} \in {(WAW)\{2\}}$,
{we have}
 \begin{equation*}
   A^{\textcircled{\dag},W}WAW=P_{\mathcal{R}((AW)^k),
        \mathcal{N}\left(((WA)^k)^{*}WAW\right)}.
 \end{equation*}
 \par
\eqref{wcoreepimaitem7}
According to Lemma \ref{wcoreepprole}\eqref{wcoreepimaitem1} and the fact that
\begin{equation}\label{WAcoreEPWWAcEpeq}
 WA((WA)^{\textcircled{\dag}})^2=(WA)^{\textcircled{\dag}},
\end{equation}
it follows    that
  \begin{equation}\label{WAWcoreepeqwacoreeqeq}
   WA^{\textcircled{\dag},W}=(WA)^{\textcircled{\dag}}
  \end{equation}
   and
 \begin{equation*}
{(A^{\textcircled{\dag},W})^{\star s}=
 A^{\textcircled{\dag},W}((WA)^{\textcircled{\dag}})^{s-1}=
 A((WA)^{\textcircled{\dag}})^{s+1}.}
 \end{equation*}
 Using
 \cite[Theorem 2.6]{pseudocorre}, i.e.,
 \begin{equation}\label{AscoreepeqAcoreseq}
  {((WA)^s)^{\textcircled{\dag}}=((WA)^{\textcircled{\dag}})^s,}
 \end{equation}
 by Lemma \ref{wcoreepprole}\eqref{wcoreepimaitem1} and \eqref{WAcoreEPWWAcEpeq},
    {we get}
 \begin{equation*}
 {  (A^{\star s})^{\textcircled{\dag},W}
   = A^{\star s}((WA^{\star s})^{\textcircled{\dag}})^2
 =A(WA)^{s-1}((WA)^{\textcircled{\dag}})^{ s}((WA)^{\textcircled{\dag}})^{ s}
 =  A((WA)^{\textcircled{\dag}})^{s+1}.}
  \end{equation*}
Hence, $(A^{\textcircled{\dag},W})^{\star s}=(A^{\star s})^{\textcircled{\dag},W}$.
\end{proof}

\section{Results}\label{resultsection}
In this section,  some novel properties and representations of the $W$-$m$-WG inverse are proposed.   Mosi\'{c} et al. in \cite[Lemma 2.2(ii), (iii)]{MosicWmWG} demonstrated two projectors correlated with the $W$-$m$-WG inverse, i.e.,
 \begin{equation}\label{wmwgproAXit01}
    WAWA^{\textcircled{w}_m,W}=P_{\mathcal{R}\left((WA)^k\right),
       \mathcal{N}\left(((WA)^k)^*(WA)^m\right)},
 \end{equation}
 \begin{equation*}
   A^{\textcircled{w}_m,W}WAW=P_{\mathcal{R}\left((AW)^k\right),
       \mathcal{N}\left(((WA)^k)^*(WA)^{m+1}W\right)},
 \end{equation*}
  where  $A\in\mathbb{C}^{q\times n}$, $W(\neq 0)\in\mathbb{C}^{n\times q}$, $m\in\mathbb{Z}^{+}$ and $k=\max\{{\rm Ind}(AW),{\rm Ind}(WA)\}$.
Inspired by the above result, we start this section by giving   two other  projectors with respect to  the $W$-$m$-WG inverse.

\begin{theorem}\label{wmwgproth}
Let $A\in\mathbb{C}^{q\times n}$, $W(\neq 0)\in\mathbb{C}^{n\times q}$, $m\in\mathbb{Z}^{+}$ and $k=\max\{{\rm Ind}(AW),{\rm Ind}(WA)\}$. Then:
 \begin{enumerate}[$(1)$]
   \item\label{wmwgproXAit03}
   $AWA^{\textcircled{w}_m,W}W=
      P_{\mathcal{R}\left((AW)^{k}\right),
       \mathcal{N}\left((WA)^{k})^*W(AW)^{ m}\right)}$;
   \item\label{wmwgproXAit04} $WA^{\textcircled{w}_m,W}WA=
   P_{\mathcal{R}\left( (WA)^{k}\right),
       \mathcal{N}\left( ((WA)^{k})^*(WA)^{m+1} \right)}$.
 \end{enumerate}
\end{theorem}

\begin{proof}
\eqref{wmwgproXAit03}
Since ${\rm Ind}(A^{\star m}W)  \leq {\rm Ind}(AW) $ and ${\rm Ind}(WA^{\star m}) \leq {\rm Ind}(WA)$, {we derive}
\begin{equation*}
  l:={\rm max}\{ {\rm Ind}(A^{\star m}W) ,{\rm Ind}(WA^{\star m}) \}
   \leq  k=\max\{{\rm Ind}(AW),{\rm Ind}(WA)\}.
\end{equation*}
Then, using  \cite[Theorem 2.1]{MosicWmWG}, i.e.,
\begin{equation*}
AWA^{\textcircled{w}_m,W}
    = (A^{\textcircled{\dag},W})^{\star m}WA^{\star m},
\end{equation*}
by items \eqref{wcoreepimaitem7} and \eqref{wcoreepimaitem3} in Lemma \ref{wcoreepprole},  we have
\begin{equation*}{ \begin{aligned}
    AWA^{\textcircled{w}_m,W}W&
    = (A^{\textcircled{\dag},W})^{\star m}WA^{\star m}W = (A^{\star m})^{\textcircled{\dag},W}WA^{\star m}W  \\
   &=
   P_{\mathcal{R}((A^{\star m} W)^l),
        \mathcal{N}\left(((WA^{\star m})^l)^{*}WA^{\star m}W\right)}=
       P_{\mathcal{R}\left((AW)^{k}\right),
       \mathcal{N}\left(((WA)^{k})^*W(AW)^{ m}\right)}.
\end{aligned}} \end{equation*}
\par
\eqref{wmwgproXAit04}
Using the fact that
\begin{equation*}
  {\rm Ind}((WA)^{  (m+1)})\leq {\rm Ind}(WA) \leq
k=\max\{{\rm Ind}(AW),{\rm Ind}(WA)\},
\end{equation*}
  by \eqref{Wmwgmpmaineq}, \eqref{WAWcoreepeqwacoreeqeq},  \eqref{AscoreepeqAcoreseq} and Lemma \ref{wcoreepprole}\eqref{wcoreepimaitem3} as $W=I_n$,   {we have}
 \begin{align*}
WA^{\textcircled{w}_m,W}WA&=
W(A^{\textcircled{\dag},W})^{\star (m+1)}WA^{\star m}WA
=((WA)^{\textcircled{\dag}})^{  (m+1)}(WA)^{m+1}\\
&
= ((WA)^{(m+1)})^{\textcircled{\dag}}(WA)^{m+1}
 =   P_{\mathcal{R}\left( (WA)^{k}\right),
       \mathcal{N}\left( ((WA)^{k})^*(WA)^{m+1} \right)}.
 \end{align*}
This finishes the proof.
\end{proof}

As stated in Definition \ref{vardenfth}\eqref{vardenfitem04},  the $W$-$m$-WG inverse reduces  to  the $W$-WG inverse when $m=1$. However, in the other case $m>1$,  there is an interesting relation between two projectors formed by the $W$-$m$-WG inverse in terms of Theorem \ref{wmwgproth}\eqref{wmwgproXAit04} and \cite[Lemma 2.2(ii)]{MosicWmWG}.

\begin{theorem}\label{WAWAWmWeqWaWmWAeq}
Let $A\in\mathbb{C}^{q\times n}$, $W(\neq 0)\in\mathbb{C}^{n\times q}$, $m\in\mathbb{Z}^{+}$ and $m>1$.
Then,
\begin{equation*}
  WAWA^{\textcircled{w}_m,W}=WA^{\textcircled{w}_{m-1},W}WA.
\end{equation*}
\end{theorem}

\begin{proof}
Let $k=\max\{{\rm Ind}(AW),{\rm Ind}(WA)\}$. Since $m>1$, it follows from Theorem \ref{wmwgproth}\eqref{wmwgproXAit04} that
\begin{equation*}
  WA^{\textcircled{w}_{m-1},W}WA=
   P_{\mathcal{R}\left( (WA)^{k}\right),
       \mathcal{N}\left( ((WA)^{k})^*(WA)^{m} \right)}.
\end{equation*}
Then, by  \eqref{wmwgproAXit01}  {we have}
\begin{equation*}
WAWA^{\textcircled{w}_m,W}=P_{\mathcal{R}((WA)^k),\mathcal{N}\left((WA)^k)^*(WA)^{m}\right)}= WA^{\textcircled{w}_{m-1},W}WA,
\end{equation*}
which completes the proof.
\end{proof}

\begin{remark}
Under the hypotheses of Theorem \ref{WAWAWmWeqWaWmWAeq},
if $W=I_n$,  by Theorem \ref{WAWAWmWeqWaWmWAeq} {we have}
\begin{equation*}
  AA^{\textcircled{w}_m}=A^{\textcircled{w}_{m-1}}A,
\end{equation*}
 which extends \cite[Theorem 3.3(d)]{GGinverseFerryre}, i.e., $AA^{\textcircled{w}_2}=A^{\textcircled{w}}A$.
\end{remark}

We then   give some  explicit expressions for the $W$-$m$-WG inverse by  using the $W$-weighted Drazin inverse,   Moore-Penrose inverse,    $W$-weighted weak group inverse, $W$-weighted group inverse,   $m$-weak group inverse,  and    weighted core inverse,  respectively.

 \begin{theorem}\label{represwmwgi01th}
 Let $A\in\mathbb{C}^{q\times n}$, $W(\neq 0)\in\mathbb{C}^{n\times q}$, $m\in\mathbb{Z}^{+}$ and  $k=\max\{{\rm Ind}(AW),{\rm Ind}(WA)\}$. Then:
\begin{enumerate}[$(1)$]
  \item\label{represwmwgi01item01} $A^{\textcircled{w}_m,W}=(A^{D,W})^{\star (m+1)}P_{\mathcal{R}((WA)^k)}WA^{\star m}$;
     \item\label{represwmwgi01item02}
$A^{\textcircled{w}_m,W}=A^{\star l}W(WA^{\star (l+m+1)}W)^{\dag}WA^{\star m}
    $, where  $l\geq  k$;
  \item\label{represwmwgi01item03}  $A^{\textcircled{w}_m,W}=(WA^{\star (m+1)}WP_{\mathcal{R}((AW)^k)})^{\dag}WA^{\star m}$;
  \item\label{represwmwgi01item04} $A^{\textcircled{w}_m,W}=A^{\star (m-1)}W(A^{\star m})^{\textcircled{w},W}$;
  \item\label{represwmwgi01item06}
  $A^{\textcircled{w}_m,W}=(A^{\textcircled{w},W})^{\star m}WA^{\star (m-1)}$;
  \item\label{represwmwgi01item001}
$  A^{\textcircled{w}_m,W} = \left((AW)^mA^{\textcircled{\dag},W}WA\right)^{\#,W}WA^{\star (m-1)}$;
  \item\label{represwmwgi01item05} $A^{\textcircled{w}_m,W}=A((WA)^{\textcircled{w}_m})^2$.
\end{enumerate}

\end{theorem}

\begin{proof}
\eqref{represwmwgi01item01}
Note that
\begin{equation*}
    {\rm rank}((WA)^k)={\rm rank}((WA)^{k+1})
  \leq  {\rm rank}(W(AW)^{k})\leq  {\rm rank}((WA)^{k}),
\end{equation*}
implying that
   \begin{equation*}
    {\rm rank}(W(AW)^{k})={\rm rank}((WA)^{k}).
   \end{equation*}
Then, using  \cite[Lemma 2(c)]{wdraziWeire}, i.e.,
 $\mathcal{R}(A^{D,W})=\mathcal{R}((AW)^k)$,
 we have
  \begin{equation*}
    \mathcal{R}(WA^{D,W})=\mathcal{R}(W(AW)^k)=\mathcal{R}{((WA)^k)}.
  \end{equation*}
Thus, by \eqref{Wmwgmpmaineq} and
 \cite[Theorem 4.1]{wcoreepigaore}, i.e.,
  $ A^{\textcircled{\dag},W}=A^{D,W}P_{\mathcal{R}((WA)^k)}$,
it follows that
 \begin{align*}
 A^{\textcircled{w}_m,W} &
 = (A^{\textcircled{\dag},W})^{\star (m+1)}WA^{\star m}
 =A^{D,W}\left(P_{\mathcal{R}((WA)^k)}WA^{D,W}\right)^m P_{\mathcal{R}((WA)^k)}WA^{\star m}\\
  &=(A^{D,W})^{\star (m+1)} P_{\mathcal{R}((WA)^k)}WA^{\star m}.
 \end{align*}
 \par
 \eqref{represwmwgi01item02}
Let  $l\geq  k$. Since
\begin{equation*}
  \mathcal{N}(W(AW)^{l+1}) \subseteq   \mathcal{N}((AW)^{l+2}) = \mathcal{N}((AW)^{l}),
\end{equation*}
 using  \cite[Theorem 2.10]{wcoreepigaore}, i.e.,
 $A^{\textcircled{\dag},W}=(AW)^l(W(AW)^{l+1})^{\dag}$, and
 the fact $B^{\dag}B=P_{\mathcal{R}({B^*})}$ for $B\in\mathbb{C}^{q\times n}$
yields that
\begin{equation*}
   A^{\textcircled{\dag},W} W A^{\textcircled{\dag},W}
 =  (AW)^l(W(AW)^{l+1})^{\dag} W (AW)^{l+1}(W(AW)^{l+2})^{\dag}=
(AW)^l(W(AW)^{l+2})^{\dag}.
\end{equation*}
Then, reusing the method similar to that used above,  we have
\begin{align*}
(A^{\textcircled{\dag},W})^{\star(m+1)}=
  A^{\star l}W(WA^{\star (l+m+1)}W)^{\dag}.
\end{align*}
Evidently, the item \eqref{represwmwgi01item02} follows by \eqref{Wmwgmpmaineq}.
  \par
  \eqref{represwmwgi01item03}
 Let $d={\rm Ind}((AW)^{m+1})$. It can easily be verified that   $d(m+1)\geq {\rm Ind}(AW)$,  which implies
 \begin{equation*}
    P_{\mathcal{R}((A^{\star (m+1)}W)^d)}=
    P_{\mathcal{R}((AW)^{ d(m+1)})}=P_{\mathcal{R}((AW)^{k})}.
 \end{equation*}
Thus, according to Lemma   \ref{wcoreepprole}\eqref{wcoreepimaitem7} and \cite[Theorem 5.2]{wcoreepiFerre}, i.e.,
 $A^{\textcircled{\dag},W}=
 \left(WAWP_{\mathcal{R}((AW)^k)}\right)^{\dag}$, we  derive
 \begin{align}
    (A^{\textcircled{\dag},W})^{\star (m+1)} &= (A^{\star (m+1)})^{\textcircled{\dag},W}=
    \left(WA^{\star (m+1)}WP_{\mathcal{R}((A^{\star (m+1)}W)^d)}\right)^{\dag}\nonumber \\
    &=
    \left(WA^{\star (m+1)}W P_{\mathcal{R}((AW)^{k})}\right)^{\dag}\label{Awcormj1eqeq}.
 \end{align}
 Substituting \eqref{Awcormj1eqeq} into \eqref{Wmwgmpmaineq} yields
 $A^{\textcircled{w}_m,W}=
    (WA^{\star (m+1)}WP_{\mathcal{R}((AW)^k)})^{\dag}WA^{\star m}
     $.
 \par
\eqref{represwmwgi01item04}
It follows from \eqref{WAWcoreepeqwacoreeqeq}  and Lemma \ref{wcoreepprole}\eqref{wcoreepimaitem1} that
\begin{equation}\label{AcoreEPWciriceq}
  AWA^{\textcircled{\dag},W}WA^{\textcircled{\dag},W}
  = A((WA)^{\textcircled{\dag}})^2=A^{\textcircled{\dag},W},
\end{equation}
which implies
\begin{equation}\label{ecoreepimm1eq}
  A^{\star (m-1)}W(A^{\textcircled{\dag},W})^{\star m}=A^{\textcircled{\dag},W}.
\end{equation}
Using \cite[Corollary 2]{wweakgroupFerre}, i.e.,
\begin{equation}\label{BWWBcoreEPWBeq}
B^{\textcircled{w},W}=(B^{\textcircled{\dag},W})^{\star 2}WB
\end{equation}
 for $B\in\mathbb{C}^{q\times n}$, by \eqref{Wmwgmpmaineq}, \eqref{ecoreepimm1eq}  and Lemma \ref{wcoreepprole}\eqref{wcoreepimaitem7},
 we have
\begin{align*}
A^{\textcircled{w}_m,W}&
=A^{\textcircled{\dag},W}W(A^{\textcircled{\dag},W})^{\star m}WA^{\star m}
=A^{\star (m-1)}W((A^{\textcircled{\dag},W})^{\star m})^{\star 2}WA^{\star m}\\
&=
A^{\star (m-1)}W((A^{\star m})^{\textcircled{\dag},W})^{\star 2}WA^{\star m}
=A^{\star (m-1)}W(A^{\star m})^{\textcircled{w},W}.
\end{align*}
\par
\eqref{represwmwgi01item06}
In terms of  \eqref{Wmwgmpmaineq}, \eqref{AcoreEPWciriceq} and  \eqref{BWWBcoreEPWBeq}, we  get
\begin{align*}
A^{\textcircled{w}_m,W}
 &= (A^{\textcircled{\dag},W})^{\star (m+1)}WA^{\star m}
=
(A^{\textcircled{\dag},W})^{\star 2} W(AW(A^{\textcircled{\dag},W})^{\star 2})^{\star (m-1)}WAWA^{\star (m-1)}
\\&=((A^{\textcircled{\dag},W})^{\star 2}WA)^{\star m}WA^{\star (m-1)}
=(A^{\textcircled{w},W})^{\star m}WA^{\star (m-1)}.
\end{align*}
\par
\eqref{represwmwgi01item001}
From Lemma \ref{wcoreepprole}\eqref{wcoreepimaitem3} and \cite[Proposition 3.1]{wcoreepigaore}, i.e., $\mathcal{R}(A^{\textcircled{\dag},W})=\mathcal{R}((AW)^k)$,  we obtain
  \begin{equation}\label{AWAWAAWcoreepeq}
   A^{\textcircled{\dag},W}WAW AW  A^{\textcircled{\dag},W}=AW  A^{\textcircled{\dag},W}.
  \end{equation}
By the fact that
$
 (B^m)^{\#}={(B^{\#})}^m
$
for $B\in\mathbb{C}^{n\times n} $ and ${\rm Ind}(B) \leq 1 $,   and   \cite[Corollary 2.1]{wdrazinclinere}, i.e.,
$ C^{\#,W}=C((WC)^{\#})^2$ for $C\in\mathbb{C}^{q\times n}$ satisfying ${\max}\{{\rm Ind}(CW),{\rm Ind}(WC)\} \leq 1$,
it can be verified that
\begin{equation}\label{conWgroupeq}
  (C^{\#,W})^{\star m}=   (C ^{\star m})^{\#,W}.
\end{equation}
Then, using  item \eqref{represwmwgi01item06} and \cite[Theorem 7]{wweakgroupFerre}, i.e.,
$
  A^{\textcircled{w},W}=(AWA^{\textcircled{\dag},W}WA)^{\#,W},
$
by \eqref{conWgroupeq}  and \eqref{AWAWAAWcoreepeq}  we have
\begin{align*}
A^{\textcircled{w}_m,W} =& (A^{\textcircled{w},W})^{\star m}WA^{\star (m-1)}=
 ((AWA^{\textcircled{\dag},W}WA)^{\#,W})^{\star m}WA^{\star (m-1)}\\
 &=((AWA^{\textcircled{\dag},W}WA)^{\star m})^{\#,W}WA^{\star (m-1)}=
 \left((AW)^mA^{\textcircled{\dag},W}WA\right)^{\#,W}WA^{\star (m-1)}.
\end{align*}

\eqref{represwmwgi01item05}
It follows  from \eqref{WAWcoreepeqwacoreeqeq} and Lemma \ref{wcoreepprole}\eqref{wcoreepimaitem1} that
\begin{align*}
(A^{\textcircled{\dag},W})^{\star (m+1)}= A^{\textcircled{\dag},W} (WA^{\textcircled{\dag},W})^{m}= A^{\textcircled{\dag},W} ((WA)^{\textcircled{\dag}})^{m}=
A((WA)^{\textcircled{\dag}})^{m+2}.
\end{align*}
Then,  using \cite[Theorem 3.1]{wmweakgroupJiare}, i.e.,
\begin{equation}\label{mweakgroum11eq}
B^{\textcircled{w}_m}=(B^{\textcircled{\dag}})^{m+1}B^{ m}
\end{equation}
 for $B\in\mathbb{C}^{n\times n}$, by  \eqref{Wmwgmpmaineq} and  \eqref{AcoreEPWciriceq} as $W=I_n$,
we obtain
\begin{align*}
 A^{\textcircled{w}_m,W}
& =(A^{\textcircled{\dag},W})^{\star (m+1)}WA^{ \star m}
  =  A((WA)^{\textcircled{\dag}})^{m+2}(WA)^{ m}
\\&=
 A((WA)^{\textcircled{\dag}})^{m+1}(WA)^{ m}((WA)^{\textcircled{\dag}})^{m+1}(WA)^{ m}
  = A((WA)^{\textcircled{w}_m})^2.
\end{align*}
 This finishes the proof.
\end{proof}

\begin{theorem}\label{corechawmwgth}
 Let $A\in\mathbb{C}^{q\times n}$, $W(\neq 0)\in\mathbb{C}^{n\times q}$, $m\in\mathbb{Z}^{+}$ and  $k=\max\{{\rm Ind}(AW),{\rm Ind}(WA)\}$.  Then:
 \begin{enumerate}[$(1)$]
    \item\label{corechawmwgitem02}  $A^{\textcircled{w}_m,W}=A^{\star (k-m)}W(A^{\star (k+1)})^{\textcircled{\#},W}WA^{\star m}$ for $ k \geq m$;
    \item \label{corechawmwgitem03}  $A^{\textcircled{w}_m,W}=A^{\star (k-m-1)}W(A^{\star k})^{\textcircled{\#},W}WA^{\star m}$ for $ k \geq  m +1$.
 \end{enumerate}
\end{theorem}

\begin{proof}
Obviously,
\begin{equation*}
   {\rm rank}((WA)^k)  =  {\rm rank}((WA)^{k+1})= {\rm rank}(W(AW)^{k}A) \leq {\rm rank}((AW)^{k}).
\end{equation*}
Similarly, $ {\rm rank}((AW)^{k}) \leq {\rm rank}((WA)^k) $. So, $ {\rm rank}((WA)^k)= {\rm rank}((AW)^{k}) $.
Hence,
\begin{align*}
 {\rm rank}((WA)^{ 2(k+1)})=&  {\rm rank}((WA)^{ (k+1)}) =
  {\rm rank}((WA)^k)\\
  =&{\rm rank}((AW)^k)=
   {\rm rank}((AW)^{ (k+1)})=   {\rm rank}((AW)^{ 2(k+1)}),
\end{align*}
which implies
 ${\rm Ind} (WA^{\star (k+1)}) = {\rm Ind} (A^{\star (k+1)}W)\leq 1$,  that is,  $((WA)^{k+1})^{\textcircled{\#}}$ and $(A^{\star (k+1)})^{\textcircled{\#},W}$ exist.

\eqref{corechawmwgitem02}
For $ k \geq m$,  it follows from
Lemma \ref{wcoreepprole}\eqref{wcoreepimaitem7}, \eqref{ecoreepimm1eq} and \eqref{Wmwgmpmaineq}  that
 \begin{align*}
A^{\star (k-m)}W(A^{\star (k+1)})^{\textcircled{\#},W}WA^{\star m} &
= A^{\star (k-m)}W(A^{\star (k+1)})^{\textcircled{\dag},W}WA^{\star m}\\
  &= A^{\star (k-m)}W(A^{\textcircled{\dag},W})^{\star (k-m+1)}W(A^{\textcircled{\dag},W})^{\star m}
  WA^{\star m}\\
 &=  A^{\textcircled{\dag},W}W(A^{\textcircled{\dag},W})^{\star m}
  WA^{\star m}=A^{\textcircled{w}_m,W}.
\end{align*}

\eqref{corechawmwgitem03}
If $k=m+1$, then by Lemma \ref{wcoreepprole}\eqref{wcoreepimaitem7} and \eqref{Wmwgmpmaineq} we obtain
\begin{align*}
  A^{\star (k-m-1)}W(A^{\star k})^{\textcircled{\#},W}WA^{\star m} &=(A^{\star k})^{\textcircled{\dag},W}WA^{\star m}=(A^{\textcircled{\dag},W})^{\star k}WA^{\star m}\\
  &=(A^{\textcircled{\dag},W})^{\star (m+1)}WA^{\star m}=A^{\textcircled{w}_m,W}.
  \end{align*}
Thus, we only need to consider the other case, i.e.,   $k>m+1$.
By item \eqref{corechawmwgitem02} and \eqref{WAWcoreepeqwacoreeqeq}, we have
\begin{align}\label{wcoreinviseq}
 A^{\textcircled{w}_m,W}&=A^{\star (k-m)}W(A^{\star (k+1)})^{\textcircled{\#},W}WA^{\star m} = A^{\star (k-m)}((WA)^{ k+1})^{\textcircled{\#}}WA^{\star m}.
\end{align}
By the core-nilpotent decomposition of $WA$, for $s\in\mathbb{Z}^{+} $ it is easy  to check that
\begin{equation}\label{groupconkj}
  (WA)^s((WA)^{k+1})^{\#}=  (WA)^{s-1}((WA)^{k})^{\#}.
\end{equation}
Using \cite[Theorem 1(1)]{Bakscoreinversere}, i.e., $B^{\textcircled{\#}}=B^{\#}P_{\mathcal{R}(B)}$ for
$B\in\mathbb{C}^{n\times n} $ and ${\rm Ind}(B) \leq 1 $,
by  \eqref{wcoreinviseq}, \eqref{groupconkj} and \eqref{WAWcoreepeqwacoreeqeq} we get
\begin{align*}
A^{\textcircled{w}_m,W}&=A^{\star (k-m)}((WA)^{k+1})^{\textcircled{\#}}WA^{\star m}=
A^{\star (k-m)}((WA)^{k+1})^{{\#}}P_{\mathcal{R}((WA)^{k+1})}WA^{\star m}\\
&
=A(WA)^{ k-m-1}((WA)^{k+1})^{{\#}}P_{\mathcal{R}((WA)^{k+1})}WA^{\star m}\\&=A^{\star (k-m-1
)}((WA)^{k})^{{\#}}P_{\mathcal{R}((WA)^{k})}WA^{\star m}
=A^{\star (k-m-1)}W(A^{\star k})^{\textcircled{\#},W}WA^{\star m},
\end{align*}
which completes the proof.
\end{proof}

\begin{remark}
\begin{enumerate}[$(1)$]
  \item
The  items
\eqref{represwmwgi01item01}, \eqref{represwmwgi01item03}, \eqref{represwmwgi01item02} and \eqref{represwmwgi01item06} of
Theorem \ref{represwmwgi01th}, and   Theorem \ref{corechawmwgth}\eqref{corechawmwgitem02}  extend  \cite[Theorem 5.1(a),(d)]{wmweakgroupJiare},  \cite[Theorem 2.5]{MosicWmWGapp}, \cite[Theorem 3.3(b)]{GGinverseFerryre}, and  \cite[Theorem 5.1(b)]{wmweakgroupJiare}, respectively.
  \item If $W=I_n$, by Theorem \ref{represwmwgi01th}\eqref{represwmwgi01item04} we have
\begin{equation*}
  A^{\textcircled{w}_m}=A^{m-1}(A^m)^{\textcircled{w}},
\end{equation*}
which simplifies  \cite[Theorem 5.1(e)]{wmweakgroupJiare}, i.e.,
$A^{\textcircled{w}_m}=A^{m-1}P_{\mathcal{R}{(A^k)}}(A^m)^{\textcircled{w}}$.
 \item Furthermore,    by Theorem \ref{corechawmwgth}\eqref{corechawmwgitem03}, for $ k  \geq  m +1$ and $W=I_n$ we have
\begin{equation}\label{AmWGiathereq}
A^{\textcircled{w}_m}=A^{k-m-1}(A^{ k})^{\textcircled{\#}}A^{m}.
\end{equation}
 It can be found that \eqref{AmWGiathereq}  is just a simplified form of  \cite[Theorem 5.1(c)]{wmweakgroupJiare}, i.e,  for $ k \geq  m +1$,
 \begin{equation*}
 A^{\textcircled{w}_m}=(A^{ k})^{\#} A^{k-m-1}P_{\mathcal{R}(A^k)}A^{m}.
 \end{equation*}
 \end{enumerate}
\end{remark}

The next result is to represent  the $W$-$m$-WG inverse  through three concise  formulas composed  of  the $m$-weak group inverse of $WA$   along with   three  different   generalized inverses of $AW$:  the $m$-weak group inverse,    core-EP inverse and   Drazin inverse.

\begin{theorem}\label{conciseforth}
  Let $A\in\mathbb{C}^{q\times n}$, $W(\neq 0)\in\mathbb{C}^{n\times q}$ and $m\in\mathbb{Z}^{+}$. Then:
    \begin{enumerate}[$(1)$]
    \item\label{WmAWAWmchaWmWgitem01} $A^{\textcircled{w}_m,W}=(AW)^{\textcircled{w}_m}A(WA)^{\textcircled{w}_m}$;
    \item\label{WmAWAWmchaWmWgitem02} $A^{\textcircled{w}_m,W}=(AW)^{\textcircled{\dag}}A(WA)^{\textcircled{w}_m}$;
    \item\label{WmAWAWmchaWmWgitem03}
     $A^{\textcircled{w}_m,W}=(AW)^{D}A(WA)^{\textcircled{w}_m}$.
    \end{enumerate}
\end{theorem}

\begin{proof}
Let  $k=\max\{{\rm Ind}(AW),{\rm Ind}(WA)\}$. Using the fact that $ BB^{\dag}=P_{\mathcal{R}(B)}$ for $B\in\mathbb{C}^{q\times n}$, and
\cite[Theorem 2.3]{pseudocorre}, i.e.,
\begin{equation}\label{coreEPBDDdDdMPeq}
C^{\textcircled{\dag}}=C^DC^d(C^d)^{\dag},
\end{equation}
where
$C\in\mathbb{C}^{n\times n}$ and $d\geq{\rm Ind}(C)$,
by \eqref{Wmwgmpmaineq} and Lemma \ref{wcoreepprole}\eqref{wcoreepimaitem1}
{we deduce}
\begin{equation*}{
 \begin{aligned}
  A^{\textcircled{w}_m,W}&=(A^{\textcircled{\dag},W})^{\star (m+1)}WA^{\star m}
  = A((WA)^{\textcircled{\dag}})^2W(A^{\textcircled{\dag},W})^{\star m}WA^{\star m}\\
  &=A(WA)^D(WA)^k((WA)^k)^{\dag}(WA)^{\textcircled{\dag}}W(A^{\textcircled{\dag},W})^{\star m}WA^{\star m}\\
  &=A(WA)^D(WA)^{\textcircled{\dag}}W(A^{\textcircled{\dag},W})^{\star m}WA^{\star m}
  =A^{\star k}(WA)^D((WA)^k)^{\dag}W(A^{\textcircled{\dag},W})^{\star m}WA^{\star m}.
 \end{aligned}}
\end{equation*}
Moreover, in terms of \eqref{mweakgroum11eq}, Lemma \ref{wcoreepprole}\eqref{wcoreepimaitem3} as $W=I_n$, \eqref{coreEPBDDdDdMPeq}   and \eqref{WAWcoreepeqwacoreeqeq}, it follows that
\begin{equation*}{
  \begin{aligned}
  (AW)^{\textcircled{w}_m}A(WA)^{\textcircled{w}_m}
 =&  ((AW)^{\textcircled{\dag}})^{m+1}(AW)^{ m}A(WA)^{\textcircled{w}_m}
=((AW)^{\textcircled{\dag}})^{m}(AW)^{\textcircled{\dag}}AW(AW)^{ m-1}A(WA)^{\textcircled{w}_m}
\\=& ((AW)^{\textcircled{\dag}})^{m} (AW)^{ m-1}A(WA)^{\textcircled{w}_m}
=...=  (AW)^{\textcircled{\dag}}A(WA)^{\textcircled{w}_m}
\\=& (AW)^D(AW)^k((AW)^k)^{\dag}A(WA)^{\textcircled{w}_m}
 =  (AW)^DA(WA)^{\textcircled{w}_m}
\\=&(AW)^DA ((WA)^{\textcircled{\dag}})^{m+1}(WA)^{ m}
= A^{\star k} (WA)^D((WA)^k)^{\dag}W(A^{\textcircled{\dag},W})^{\star m}W
A^{\star m}.
  \end{aligned}}
\end{equation*}
Consequently, the items \eqref{WmAWAWmchaWmWgitem01}--\eqref{WmAWAWmchaWmWgitem03} hold.
\end{proof}

 Mosi\'{c} et al. in \cite[Lemma 2.1]{MosicWmWGapp} have proposed a block representation for the
$W$-$m$-WG inverse by the weighted core-EP decomposition introduced by \cite[Theorem 4.1]{wcoreepiFerre}.  Inspired by \cite[Theorem 2.4]{wcoreepigaore}, we now show
a new  canonical form of the $W$-$m$-WG inverse from another perspective.
  Let   the
singular value decompositions of  $A\in\mathbb{C}^{q\times n}$  with   ${\rm rank}(A)=r_1$ and $W(\neq 0)\in\mathbb{C}^{n\times q}$ with   ${\rm rank}(W)=r_2$ be
\begin{equation*}
  A=T\left(
       \begin{array}{cc}
         \Sigma_1 & 0 \\
         0 & 0 \\
       \end{array}
     \right)U^* \text{ and }W=S\left(
       \begin{array}{cc}
         \Sigma_2 & 0 \\
         0 & 0 \\
       \end{array}
     \right)V^*,
\end{equation*}
  where   $T,V \in\mathbb{C}^{q\times q}$ and  $U,S\in\mathbb{C}^{n\times n}$ are  unitary matrices, and  $ \Sigma_1 \in\mathbb{C}^{r_1\times r_1}$ and  $\Sigma_2 \in\mathbb{C}^{r_2\times r_2}$ are nonsingular matrices. Denote
  \begin{equation}\label{KLHReq}
  \left(
           \begin{array}{cc}
             K_1 & L_1 \\
              H_1 & R_1 \\
           \end{array}
         \right)=  U^*S
         \text{ and }
    \left(
           \begin{array}{cc}
             K_2 & L_2 \\
              H_2 & R_2 \\
           \end{array}
         \right)=   V^*T,
  \end{equation}
 where  $K_1\in\mathbb{C}^{r_1\times r_2}$, $L_1\in\mathbb{C}^{r_1\times(n- r_2)}$,  $H_1\in\mathbb{C}^{(n-r_1)\times r_2}$, $R_1\in\mathbb{C}^{(n- r_1)\times(n- r_2)}$,
  $K_2 \in\mathbb{C}^{r_2\times r_1}$,  $L_2 \in\mathbb{C}^{r_2\times(q- r_1)}$,
  $H_2 \in\mathbb{C}^{(q-r_2)\times r_1}$ and   $R_2 \in\mathbb{C}^{(q-r_2)\times(q- r_1)}$ are such that $K_1K_1^*+L_1L_1^*=I_{r_1}$ and $K_2K_2^*+L_2L_2^*=I_{r_2}$.
Hence,  \begin{equation}\label{wHSdecAWeq}
  A=T\left(
      \begin{array}{cc}
     \Sigma_1 K_1     &       \Sigma_1 L_1  \\
      0   & 0 \\
      \end{array}
    \right)S^*
    \text{ and }
    W=S\left(
      \begin{array}{cc}
    \Sigma_2 K_2     &       \Sigma_2 L_2  \\
      0   & 0 \\
      \end{array}
    \right)  T^*.
\end{equation}

\begin{theorem}\label{WmWGinHSdecth}
Let $A\in\mathbb{C}^{q\times n}$ and  $W(\neq 0)\in\mathbb{C}^{n\times q}$ be given by \eqref{wHSdecAWeq}, and  let  $m\in\mathbb{Z}^{+}$. Then
\begin{equation}\label{WmWGinHSdeceq}
  A^{\textcircled{w}_m,W}=T\left(
    \begin{array}{cc}
( \Sigma_1K_1)^{\textcircled{w}_m, \Sigma_2K_2}
 & B
    \\
      0 & 0 \\
    \end{array}
  \right)S^*,
\end{equation}
where $ B=( \Sigma_1K_1)^{\textcircled{\dag}, \Sigma_2K_2}
(\Sigma_2K_2( \Sigma_1K_1)^{\textcircled{\dag}, \Sigma_2K_2})^{m}(\Sigma_2 K_2  \Sigma_1 K_1)^{m-1}      \Sigma_2 K_2 \Sigma_1 L_1$.
\end{theorem}

\begin{proof}
Evidently,
\begin{align*}
  WA&=S\left(
      \begin{array}{cc}
    \Sigma_2 K_2     &       \Sigma_2 L_2  \\
      0   & 0 \\
      \end{array}
    \right)   \left(
      \begin{array}{cc}
     \Sigma_1 K_1     &       \Sigma_1 L_1  \\
      0   & 0 \\
      \end{array}
    \right)S^*
    =S\left(
      \begin{array}{cc}
    \Sigma_2 K_2  \Sigma_1 K_1     &         \Sigma_2 K_2 \Sigma_1 L_1  \\
      0   & 0 \\
      \end{array}
    \right) S^* .
\end{align*}
From Lemma \ref{wcoreepprole}\eqref{wcoreepimaitem1} and \cite[Theorem 2.4]{wcoreepigaore}, we know
\begin{equation*}
 A^{\textcircled{\dag},W}=  T
    \left(
    \begin{array}{cc}
  \Sigma_1K_1 ((\Sigma_2K_2\Sigma_1K_1)^{\textcircled{\dag}})^2 & 0
    \\
      0 & 0 \\
    \end{array}
  \right)S^*=
  T
    \left(
    \begin{array}{cc}
( \Sigma_1K_1)^{\textcircled{\dag}, \Sigma_2K_2} & 0
    \\
      0 & 0 \\
    \end{array}
  \right)S^*
  .
\end{equation*}
Then, it follows from \eqref{Wmwgmpmaineq} that
\begin{align*}
  A^{\textcircled{w}_m,W}&=(A^{\textcircled{\dag},W})^{\star (m+1)}WA^{\star m}=
  A^{\textcircled{\dag},W}(WA^{\textcircled{\dag},W})^{m}(WA)^m
  \\&=
    T
    \left(
    \begin{array}{cc}
( \Sigma_1K_1)^{\textcircled{\dag}, \Sigma_2K_2}
(\Sigma_2K_2( \Sigma_1K_1)^{\textcircled{\dag}, \Sigma_2K_2})^{m} & 0
    \\
      0 & 0 \\
    \end{array}
  \right)\\
  & ~~~~~~
  \left(
      \begin{array}{cc}
    (\Sigma_2 K_2  \Sigma_1 K_1)^m     &     (\Sigma_2 K_2  \Sigma_1 K_1)^{m-1}      \Sigma_2 K_2 \Sigma_1 L_1  \\
      0   & 0 \\
      \end{array}
    \right) S^*\\
    &=T\left(
    \begin{array}{cc}
( \Sigma_1K_1)^{\textcircled{w}_m, \Sigma_2K_2}
 & B
    \\
      0 & 0 \\
    \end{array}
  \right)S^*,
\end{align*}
which completes the proof.
\end{proof}

\begin{remark}
 Under the hypotheses of Theorem \ref{WmWGinHSdecth},  in view of Definition \ref{vardenfth}\eqref{vardenfitem04} and \eqref{WmWGinHSdeceq},   we can deduce that
 if  $m=1$, then
  \begin{equation*}
  A^{\textcircled{w},W}=
     T  \left(
    \begin{array}{cc}
      (\Sigma_1K_1)^{\textcircled{w},\Sigma_2K_2} &
        (\Sigma_1K_1)^{\textcircled{\dag},\Sigma_2K_2}\Sigma_2K_2  (\Sigma_1K_1)^{\textcircled{\dag},\Sigma_2K_2}\Sigma_2K_2\Sigma_1L_1 \\
     0& 0 \\
    \end{array}
  \right)S^*,
\end{equation*}
and if $m>1$, then
  \begin{equation*}
  A^{\textcircled{w}_m,W}=
     T  \left(
    \begin{array}{cc}
      (\Sigma_1K_1)^{\textcircled{w}_m,\Sigma_2K_2} &
        (\Sigma_1K_1)^{\textcircled{\dag},\Sigma_2K_2}\Sigma_2K_2  (\Sigma_1K_1)^{\textcircled{w}_{m-1},\Sigma_2K_2}\Sigma_2K_2\Sigma_1L_1 \\
     0& 0 \\
    \end{array}
  \right)S^*.
\end{equation*}
Moreover, if $W=I_n$ and $m\in\mathbb{Z}^{+}$, then
   \begin{equation*}
  A^{\textcircled{w}_m}=
     T  \left(
    \begin{array}{cc}
      (\Sigma_1K_1)^{\textcircled{w}_m} &
(( \Sigma_1K_1)^{\textcircled{\dag}})^{m+1}( \Sigma_1 K_1)^{m-1}   \Sigma_1 L_1 \\
     0& 0 \\
    \end{array}
  \right)T^*,
\end{equation*}
 by which we can   get  \cite[Formula (21)]{Ferreyraweakcoreinref} and \cite[Theorem 6.1]{GGinverseFerryre} immediately.
\end{remark}

\section{Numerical examples}\label{numbersection}

In this section, we give  three numerical examples to
 illustrate Theorems \ref{WAWAWmWeqWaWmWAeq}, \ref{represwmwgi01th}, \ref{corechawmwgth}, \ref{conciseforth} and \ref{WmWGinHSdecth}.
 All  numerical  computations  are performed via MATLAB R2019a on a  personal notebook computer with AMD Ryzen 7 5800H    at 3.20 GHz and 16GB of RAM.
We use ``rank($A$)", ``pinv($A$)", ``svd($A$)" and ``norm($A$,'fro')"  operators in  MATLAB to compute the rank,  Moore-Penrose inverse,   singular value decomposition and Frobenius norm of a complex matrix $A$, respectively.

\begin{example}\label{example01}
Let
\begin{equation*}
  A= {\footnotesize\left(
      \begin{array}{ccccc}
   1+1i&1&1&0&0\\0&0&0+1i&1+1i&0\\0&0&0&1+1i&0\\0&0&0&0+1i&0\\1+1i&1+1i&1&0&0\\0&0&0&0&0\\
      \end{array}
    \right)}
    \text{ and }
    W= {\footnotesize\left(
        \begin{array}{cccccc}
        0+1i&0+1i&0+1i&0+1i&1&0\\1+1i&0&0&0&1+1i&0\\0+1i&0&1+1i&0+1i&0+1i&0\\0&0&0&0&0&0\\0&0&0&0&0&0\\
        \end{array}
      \right)}.
\end{equation*}
Then, ${\rm Ind}(AW)={\rm Ind}(WA)=3$, $k={\rm max}\{{\rm Ind}(AW),{\rm Ind}(WA)\} =3$,   and
\begin{align*}
  &   A^{\textcircled{\dag},W}=A((WA)^{\textcircled{\dag}})^2 \\ &=
    {\footnotesize   \left(
      \begin{array}{ccccc}
-0.0093714-0.0086857i&-0.018743-0.017371i&-0.018057+0.00068571i&0&0
\\0.0035429-0.0019429i&0.0070857-0.0038857i&0.0016-0.0054857i&0&0\\
0&0&0&0&0\\0&0&0&0&0\\-0.0077714-0.014171i&-0.015543-0.028343i&-0.021943-0.0064i&0&0\\0&0&0&0&0\\
      \end{array}
    \right)}.
\end{align*}
If $m=1$, then
\begin{align*}
     &A^{\textcircled{w}_m,W}=  A^{\textcircled{w},W}=  (A^{\textcircled{\dag},W})^{\star 2}WA\\
&=
   {\footnotesize \left(
       \begin{array}{ccccc}
-0.015936-0.019648i&-0.018135-0.010885i&-0.016389-0.0029943i&-0.0028846-0.0092937i&0\\0.007488-0.002816i&0.0050789-0.004352i&0.0025371-0.0046171i&0.0030766+6.4e-05i&0
\\0&0&0&0&0\\0&0&0&0&0\\-0.011264-0.029952i&-0.017408-0.020315i&-0.018469-0.010149i&0.000256-0.012306i&0\\0&0&0&0&0\\
       \end{array}
     \right)}.
\end{align*}
If $m=2$, then
 \begin{align*}
   &A^{\textcircled{w}_m,W}=  (A^{\textcircled{\dag},W})^{\star 3}WA^{\star 2}\\
  & = {\footnotesize \left(
        \begin{array}{ccccc}
    -0.015936-0.019648i&-0.019318-0.011368i&-0.014723-0.0036941i&0.0018101-0.015382i&0\\0.007488-0.002816i&0.0053421-0.0046586i&0.0025805-0.0040474i&0.0044335+0.0020812i&0\\0&0&0&0&0\\0&0&0&0&0
    \\-0.011264-0.029952i&-0.018634-0.021368i&-0.016189-0.010322i&0.0083248-0.017734i&0\\0&0&0&0&0\\
        \end{array}
      \right)}.
 \end{align*}
If $m = 3$, then
\begin{align*}
&A^{\textcircled{w}_m,W}=  A^{D,W} = (A^{\textcircled{\dag},W})^{\star 4}WA^{\star 3}\\
&=
   {\footnotesize
   \left(
     \begin{array}{ccccc}
-0.015936-0.019648i&-0.019318-0.011368i&-0.014723-0.0036941i&0.0023692-0.015264i&0\\0.007488-0.002816i&0.0053421-0.0046586i&0.0025805-0.0040474i&0.0043422+0.0022371i&0\\0&0&0&0&0\\0&0&0&0&0\\-0.011264-0.029952i&-0.018634-0.021368i&-0.016189-0.010322i&0.0089485-0.017369i&0\\0&0&0&0&0\\
     \end{array}
   \right)
   }.
\end{align*}
For $m=2$,
\begin{equation*}
    \Vert WAWA^{\textcircled{w}_m,W}-WA^{\textcircled{w}_{m-1},W}WA \Vert_F=3.2605e-16,
\end{equation*}
and, for  $m=3$,
\begin{equation*}
    \Vert WAWA^{\textcircled{w}_m,W}-WA^{\textcircled{w}_{m-1},W}WA \Vert_F=3.3408e-16,
\end{equation*}
which show the validity of Theorem \ref{WAWAWmWeqWaWmWAeq}.
\end{example}

\begin{example}
Let $A$ and $W$ be given in Example \ref{example01}.
We first consider  Theorem \ref{represwmwgi01th}\eqref{represwmwgi01item01} for $m=1$.
Then,
\begin{align*}
   &(A^{D,W})^{\star (m+1)}P_{\mathcal{R}((WA)^k)}WA^{\star m}\\
   &= {\footnotesize\left(
     \begin{array}{ccccc}
-0.015936-0.019648i&-0.018135-0.010885i&-0.016389-0.0029943i&-0.0028846-0.0092937i&0\\0.007488-0.002816i&0.0050789-0.004352i&0.0025371-0.0046171i&0.0030766+6.4e-05i&0\\0&0&0&0&0\\0&0&0&0&0
\\-0.011264-0.029952i&-0.017408-0.020315i&-0.018469-0.010149i&0.000256-0.012306i&0\\0&0&0&0&0\\
     \end{array}
   \right)}.
\end{align*}
Moreover,
\begin{equation*}
E=A^{\textcircled{w}_m,W} - (A^{D,W})^{\star (m+1)}P_{\mathcal{R}((WA)^k)}WA^{\star m}
\end{equation*}
is called an error matrix for computing the  $W$-$m$-WG inverse.
Thus,  the norm of the the error matrix $E$ is
 \begin{equation*}
 \Vert E \Vert_F=  6.4874\text{e}-17 .
  \end{equation*}
Using  a similar  method   to that  described above, we can test   Theorem \ref{represwmwgi01th}\eqref{represwmwgi01item01}-\eqref{represwmwgi01item05}, Theorem \ref{corechawmwgth}\eqref{corechawmwgitem02}-\eqref{corechawmwgitem03}  and Theorem \ref{conciseforth}\eqref{WmAWAWmchaWmWgitem01}-\eqref{WmAWAWmchaWmWgitem03} in  the cases:  $m=1$,  $m=2$ and $m=3$.
Put  $l=2k$ for Theorem \ref{represwmwgi01th}\eqref{represwmwgi01item02}.
All norms of  error matrices for computing the  $W$-$m$-WG inverse by different  expressions are shown in  Table \ref{tablenorm01}.
\begin{table*}[ht]
 \centering
\setlength{\tabcolsep}{6mm}{
\begin{tabular}{|c|c|c|c|}
  \hline
  & $m=1$ & $m=2$ & $m = 3$ \\  \hline
   \text{Theorem \ref{represwmwgi01th}\eqref{represwmwgi01item01}} &$ 6.4874e-17$&$1.2023e-16$&$1.4814e-16$\\ \hline
 \text{Theorem \ref{represwmwgi01th}\eqref{represwmwgi01item02}} &$ 4.012e-17$&$3.6802e-17$&$5.6061e-17$\\ \hline
\text{Theorem \ref{represwmwgi01th}\eqref{represwmwgi01item03}} &$4.0701e-17$&$5.5611e-17$&$1.0742e-16$\\ \hline
\text{Theorem \ref{represwmwgi01th}\eqref{represwmwgi01item04}} &$4.5807e-18$&$1.3326e-16$&$2.5384e-17$\\ \hline
\text{Theorem \ref{represwmwgi01th}\eqref{represwmwgi01item06}} &$ 4.5807e-18$&$2.5012e-17$&$4.901e-17$\\ \hline
\text{Theorem \ref{represwmwgi01th}\eqref{represwmwgi01item001}} &$ 1.0638e-16$&$1.0651e-16$&$1.1824e-16$\\ \hline
\text{Theorem \ref{represwmwgi01th}\eqref{represwmwgi01item05}} &$ 1.0797e-17$&$1.6188e-17$&$1.0937e-17$\\ \hline
\text{Theorem \ref{corechawmwgth}\eqref{corechawmwgitem02}} &$ 3.9238e-17$&$5.3969e-17$&$6.8215e-17$\\ \hline
\text{Theorem \ref{corechawmwgth}\eqref{corechawmwgitem03}} &$ 1.8059e-17$&$3.1985e-17$& NAN  \\ \hline
\text{Theorem \ref{conciseforth}\eqref{WmAWAWmchaWmWgitem01}} &$ 3.2352e-17$&$4.8537e-17$&$5.7766e-17$\\ \hline
\text{Theorem \ref{conciseforth}\eqref{WmAWAWmchaWmWgitem02}} &$ 1.0816e-17$&$1.0607e-17$&$1.5265e-17$\\ \hline
\text{Theorem \ref{conciseforth}\eqref{WmAWAWmchaWmWgitem03}}  &$ 1.0738e-17$&$1.2443e-17$&$1.9151e-17$\\ \hline
\end{tabular}}
\caption{Norms of  error matrices for computing the  $W$-$m$-WG inverse}\label{tablenorm01}
\end{table*}
There is a ``NAN"  in Table \ref{tablenorm01} since  the condition    for   Theorem \ref{corechawmwgth}\eqref{corechawmwgitem03}, i.e., $k \geq  m +1 $,   is not satisfied  when $k=3$ and $m=3$.  The experimental results show  the  effectiveness of all items in Theorems \ref{represwmwgi01th}, \ref{corechawmwgth} and \ref{conciseforth} for computing the  $W$-$m$-WG inverse.
\end{example}

\begin{example}
Let  $A$ and $W$ be given in Example \ref{example01}.
In order to illustrate   Theorem \ref{WmWGinHSdecth} in $m=1$,  we  obtain  that
singular value decompositions of  $A$ and $W$  are
\begin{equation*}
  A=T\left(
       \begin{array}{cc}
         \Sigma_1 & 0 \\
         0 & 0 \\
       \end{array}
     \right)U^* \text{ and }W=S\left(
       \begin{array}{cc}
         \Sigma_2 & 0 \\
         0 & 0 \\
       \end{array}
     \right)V^*,
\end{equation*}
where
\begin{equation*}
 \Sigma_1  ={\footnotesize\left(
     \begin{array}{cccc}
3.0071&0&0&0\\0&2.2773&0&0\\0&0&0.72813&0\\0&0&0&0.49125\\
     \end{array}
   \right)},~~
 \Sigma_2  = {\footnotesize\left(
               \begin{array}{ccc}
      3.131&0&0\\0&1.8283&0\\0&0&0.9244\\
               \end{array}
             \right)},
\end{equation*}
\begin{align*}
T={\tiny\left(
     \begin{array}{cccccc}
-0.50271-0.39152i&0.14048+0.086303i&-0.14263+0.39541i&0.084912-0.6188i&-1.2467e-16-3.3295e-17i&0\\0.1881-0.1881i&0.45451-0.45451i&-0.40098+0.40098i&-0.31185+0.31185i&-1.0803e-16-9.5631e-17i&0\\0.062261-0.062261i&0.41582-0.41582i&0.32471-0.32471i&0.22609-0.22609i&0.45007+0.36162i&0\\0.062261-7.1828e-19i&0.41582+2.8515e-17i&0.32471+1.138e-15i&0.22609-8.9338e-17i&-0.088455-0.81169i&0\\-0.44711-0.5583i&0.11339+0.16757i&0.12639-0.41165i&-0.26694+0.43677i&1.2467e-16+3.3295e-17i&0\\0&0&0&0&0&1\\
    \end{array}
  \right)},
\end{align*}
\begin{equation*}
   U={\footnotesize
   \left(
     \begin{array}{ccccc}
  -0.63172&0.22296&-0.044601&-0.7411&0\\-0.50152-0.16717i&0.18506+0.061686i&-0.58765-0.19588i&0.51855+0.17285i&0\\-0.37842-0.37842i&-0.088105-0.088105i&0.5284+0.5284i&0.26426+0.26426i&0\\2.1684e-18-0.18722i&1.9082e-17-0.94695i&6.6613e-16-0.23643i&-4.1633e-17-0.11107i&0\\0&0&0&0&1\\
     \end{array}
   \right)
   },
\end{equation*}
\begin{equation*}
   S={\footnotesize
   \left(
     \begin{array}{ccccc}
-0.58628&0.65568&-0.47577&0&0\\-0.46045+0.17703i&-0.25478+0.5693i&0.21628+0.56644i&0&0\\-0.63166-0.11802i&-0.1924-0.37954i&0.51324-0.37763i&0&0\\0&0&0&0&1\\0&0&0&1&0\\
     \end{array}
   \right)
   },
\end{equation*}
\begin{equation*}
   V={\tiny
   \left(
     \begin{array}{cccccc}
-0.12822+0.5926i&-0.035558+0.19734i&0.43822+0.33826i&-0.4723+0.2219i&-0.062903+0.097214i&0\\-5.5511e-17+0.18725i&-1.6653e-16-0.35864i&-5.5511e-17+0.51467i&0.62028+0.1453i&-0.15968+0.37427i&0\\-0.23944+0.35131i&-0.31283-0.46099i&0.1467-0.44905i&0.1051-0.073925i&-0.40858-0.3198i&0\\-0.037694+0.389i&-0.20759-0.2534i&-0.40851-0.040536i&-0.031171+0.17902i&0.72837-0.08878i&0\\-0.31547+0.40535i&0.32308+0.55598i&-0.076449-0.17641i&0.4723-0.2219i&0.062903-0.097214i&0\\0&0&0&0&0&1\\
     \end{array}
   \right)
   }.
\end{equation*}
Then, by \eqref{KLHReq} we get
\begin{equation*}
  K_1={\footnotesize\left(
        \begin{array}{ccc}
     0.85539-0.36013i&-0.16518-0.2573i&0.04607+0.08919i\\-0.13895+0.01591i&0.18455+0.13756i&-0.043056+0.16997i\\-0.13408+0.077182i&-0.29325-0.48334i&-0.14518-0.76123i\\0.028218+0.30712i&-0.67077+0.2898i&0.59849+0.020921i\\
        \end{array}
      \right)},
\end{equation*}
\begin{equation*}
  K_2={\footnotesize\left(
        \begin{array}{cccc}
       -0.32717+0.63906i&-0.28114-0.48652i&-0.083097+0.024907i&-0.20002-0.055405i\\-0.45048+0.31275i&0.28013+0.55045i&-0.26803+0.0005589i&-0.093821+0.41517i\\-0.30521-0.11319i&-0.10354-0.093851i&0.40138+0.593i&-0.025916-0.14236i\\
        \end{array}
      \right)},
\end{equation*}
\begin{equation*}
       L_1={\footnotesize \left(
        \begin{array}{cc}
      0&2.1684e-18+0.18722i\\0&1.9082e-17+0.94695i\\0&6.6613e-16+0.23643i\\0&-4.1633e-17+0.11107i\\
        \end{array}
      \right)},~~       L_2={\footnotesize \left(
        \begin{array}{cc}
      -0.29314-0.17969i&0\\-0.083453+0.24044i&0\\-0.02732+0.58315i&0\\
        \end{array}
      \right)}.
\end{equation*}
Then,
\begin{equation*}
   \Sigma_1K_1=
   {\footnotesize \left(
        \begin{array}{ccc}
 2.5722-1.0829i&-0.4967-0.7737i&0.13853+0.2682i\\-0.31644+0.036231i&0.42027+0.31326i&-0.098051+0.38708i\\-0.097625+0.056199i&-0.21352-0.35194i&-0.10571-0.55428i\\0.013862+0.15087i&-0.32952+0.14236i&0.29401+0.010278i\\
        \end{array}
      \right)},
\end{equation*}
\begin{equation*}
(\Sigma_1K_1)^{\textcircled{\dag}, \Sigma_2K_2}=
{\footnotesize \left(
        \begin{array}{ccc}
        -0.048933+0.011053i&-0.011592+0.018258i&0.0011722+0.0052455i\\0.0063741-0.0024796i&0.0011938-0.0026961i&-0.00026409-0.00068355i\\-0.00065373+0.0020321i&0.00041818+0.00081978i&0.00021756+7.0544e-05i
        \\0.0026909-0.0014018i&0.00039603-0.0012467i&-0.00014953-0.00028866i\\
        \end{array}
      \right)},
\end{equation*}
\begin{align*}
 & B=( \Sigma_1K_1)^{\textcircled{\dag}, \Sigma_2K_2}
(\Sigma_2K_2( \Sigma_1K_1)^{\textcircled{\dag}, \Sigma_2K_2})^{m}(\Sigma_2 K_2  \Sigma_1 K_1)^{m-1}      \Sigma_2 K_2 \Sigma_1 L_1\\
&={\footnotesize \left(
        \begin{array}{cc}
     0&0.012412+0.0097786i\\0&-0.0018711-0.0010675i\\0&0.00062683-0.00024319i\\0&-0.00087679-0.00038024i\\
        \end{array}
      \right)},
\end{align*}
\begin{equation*}
( \Sigma_1K_1)^{\textcircled{w}_m, \Sigma_2K_2}={\footnotesize \left(
        \begin{array}{ccc}
-0.055403+0.0043184i&0.00936+0.020439i&-0.0016425-0.0015768i\\0.0073826-0.0017023i&-0.0016752-0.0025666i&0.00025334+0.00017939i
\\-0.0010405+0.0021234i&0.00095141+9.9275e-05i&-9.3312e-05+2.6067e-05i\\0.0031733-0.0011079i&-0.00086292-0.0010508i&0.0001204+6.6852e-05i\\
        \end{array}
      \right)},
\end{equation*}
\begin{align*}
 &T\left(
    \begin{array}{cc}
( \Sigma_1K_1)^{\textcircled{w}_m, \Sigma_2K_2}
 & B
    \\
      0 & 0 \\
    \end{array}
  \right)S^*\\& =
  {\tiny \left(
        \begin{array}{ccccc}
        -0.015936-0.019648i&-0.018135-0.010885i&-0.016389-0.0029943i&-0.0028846-0.0092937i&0\\0.007488-0.002816i&0.0050789-0.004352i&0.0025371-0.0046171i&0.0030766+6.4e-05i&0\\-2.2885e-19+1.4714e-18i&-3.4175e-20+1.3858e-18i&4.8876e-19+1.1278e-18i&-7.5894e-19+6.0986e-19i&0\\-8.2362e-19-2.0728e-19i&-7.9967e-19+1.2763e-19i&-5.1758e-19+3.1523e-19i&-7.0473e-19-4.4723e-19i&0
        \\-0.011264-0.029952i&-0.017408-0.020315i&-0.018469-0.010149i&0.000256-0.012306i&0\\0&0&0&0&0\\
        \end{array}
      \right)}.
\end{align*}
Hence, the norm of  error matrix is
\begin{equation*}
\Vert A^{\textcircled{w}_m,W} -   T\left(
    \begin{array}{cc}
( \Sigma_1K_1)^{\textcircled{w}_m, \Sigma_2K_2}
 & B
    \\
      0 & 0 \\
    \end{array}
  \right)S^* \Vert_F=  8.1656e-17.
\end{equation*}
Analogously,  the norms of  error matrices for computing the  $W$-$m$-WG inverse by \eqref{WmWGinHSdeceq} in Theorem \ref{WmWGinHSdecth} are   $1.2248e-16$ and $1.4167e-16$  for $m=2$ and $m=3$, respectively.
So, it is also efficient to  compute the  $W$-$m$-WG inverse in terms of  Theorem
\ref{WmWGinHSdecth}.

\end{example}

\section{Conclusions}\label{Conclusionssection}

 The paper mainly  presents some  properties and representations of the $W$-$m$-WG inverse, in which it not only generalizes some results of the $W$-weighted weak group inverse,   $W$-weighted Drazin inverse and  $m$-weak group inverse  but also provides some new results  about them. In addition, we note that Mosi\'{c} et al. \cite{MosicWmWG} have introduced the weighted version for the generalized group inverse  \cite{GGinverseFerryre} and given its properties.
Naturally, all results for the $W$-$m$-WG inverse in the paper can be condensed to those of the  $W$-weighted generalized group inverse by taking $m=2$.

We  are convinced  that further explorations for the   $W$-$m$-WG inverse  may  receive more interest. The following are two potential research directions  for the $W$-$m$-WG inverse.
 \begin{enumerate}[$(1)$]
   \item  Continuity,     perturbation analysis  and iterative methods for  the   $W$-$m$-WG inverse are all interesting study directions.

   \item It is possible  to  consider  the  $W$-$m$-WG inverse
  of an operator between two Hilbert spaces as   a generalization of  the  $W$-$m$-WG inverse of a matrix.
 \end{enumerate}

\section*{Conflict of interest}
The authors declare no conflict of interest.

\end{document}